\documentclass[11pt]{article}
\usepackage{latexsym, amsmath, amssymb}

\usepackage{amsmath}
\usepackage{amsfonts}
\usepackage{amssymb}
\usepackage{amscd}
\usepackage{amsthm}
\usepackage{fancyhdr}
\theoremstyle{plain}
\newtheorem{theorem}{Theorem}[section]
\newtheorem{proposition}[theorem]{Proposition}
\newtheorem{lemma}[theorem]{Lemma}

\theoremstyle{definition}

\date{}

\newcommand\sO{{\mathcal O}}

  \def \tab#1{\kern #1 truein}
\begin{document}
\title{Weakly uniform rank two vector bundles on multiprojective spaces}
\author{Edoardo Ballico and Francesco Malaspina
\vspace{6pt}\\
{\small   Universit\`a di Trento}\\
{\small\it 38123 Povo (TN), Italy}\\
{\small\it e-mail: ballico@science.unitn.it}\\
\vspace{6pt}\\
{\small  Politecnico di Torino}\\
{\small\it  Corso Duca degli Abruzzi 24, 10129 Torino, Italy}\\
{\small\it e-mail: francesco.malaspina@polito.it}}
    \maketitle\def\thefootnote{}
\footnote{
keywords:  uniform vector bundles, splitting type, multiprojective spaces}
 \begin{abstract}Here we classify the weakly uniform rank two vector bundles on multiprojective spaces.\\
 Moreover we show that every rank $r>2$ weakly uniform vector bundle with splitting type $a_{1,1}=\dots =a_{r,s}=0$ is trivial and every rank $r>2$ uniform vector bundle with splitting type
 $a_1>\dots >a_r$, splits.
\end{abstract}

%\maketitle

 \section{Introduction}
 We denote by $\mathbb {P}^{n}$ the $n$-dimensional projective space aver an algebraic field of characteristic
zero. A rank $r$ vector bundle $E$ on $\mathbb {P}^n$ is said to be {it uniform} if there is a sequence of integers $(a_1,\dots , a_r)$ with $a_1\geq \dots\geq a_r$ such that for every line $L$ on $\mathbb P^n$, $E_{|L}\cong \oplus_{i=1}^{r}\sO(a_i)$. The sequence $(a_1,\dots , a_r)$ is called the splitting type of $E$.\\
The classification of these bundles is known in many cases: rank $E\leq n$ with $n\geq 2$ (see \cite{vdv}, \cite{sa}, \cite{ehs}); rank $E=n+1$ for $n=2$ and $n=3$ (see \cite{ele}, \cite{el2}); rank $E=5$ for $n=3$ (see \cite{be}). Nevertheless there are uniform vector bundles  (of rank $2n$) which are not homogeneous (see \cite{dr}).\\
In \cite{bn} the authors gave the notion of weakly uniform bundle on $\mathbb {P}^1\times \mathbb {P}^1$. 
For the study of rank two weakly uniform vector bundles on $(\mathbb {P}^1)^s$, see \cite{s}, \cite{ns} and \cite{bn}. \\
Here we are interested on vector bundles on multiprojective spaces.  Fix integers $s\ge 2$ and $n_i\ge 1$. Let $X:= \mathbb {P}^{n_1}\times \cdots \times \mathbb {P}^{n_s}$ be a multiprojective
space. Let $$u_i: X \to \mathbb {P}^{n_i}$$ be the projection on the $i$-th factor. For all $1 < i < j$ let $$u_{ij}: X \to  \mathbb {P}^{n_i}\times  \mathbb {P}^{n_j}$$ denote the projection
onto the product of the $i$-th factor and the $j$-th factor.
Set $\mathcal {O}:= \mathcal {O}_X$. For all integers $b_1,\dots ,b_s$ set $\mathcal {O}(b_1,\dots ,b_s):= \otimes _{i=1}^{s} u_i^{\ast}(\mathcal {O}_{\mathbb {P}^{n_i}}(b_i))$. We recall that every line bundle on $X$
is isomorphic to a unique line bundle $\mathcal {O}(b_1,\dots ,b_s)$. Set $X_i:= \prod _{j\ne i} \mathbb {P}^{n_j}$. Let $$\pi
_i: X \to X_i$$ be the projection. Hence $\pi _i^{-1}(P) \cong \mathbb {P}^{n_i}$ for each $P\in X_i$. Let $E$ be a rank $r$
vector bundle on $X$. We say that $E$ is {\it weakly uniform} with splitting type $(a_{h,i})$, $1 \le h \le r$, $1 \le i \le
s$, if for all $i\in \{1,\dots ,s\}$, every $P\in X_i$ and every line $D\subseteq \pi _i^{-1}(P)$ the vector bundle
$E\vert D$ on $D\cong \mathbb {P}^1$ has splitting type $a_{1,i} \ge \cdots \ge a_{r,i}$.
A weakly uniform vector bundle $E$
on $X$ is called {\it uniform} if there is a line bundles $(a_1,\dots ,a_s)$ such that the splitting types of $E(a_1,\dots
,a_s)$ with respect to all $\pi _i$ are the same. In this case a splitting type of $E$ is the splitting type $c_1\ge \cdots \ge c_r$, $r:= \mbox{rank}(E)$,
of $E(a_1,\dots ,a_s)$. Notice that the $r$-ple of integers $(c_1,\dots ,c_r)$ is not uniquely determined by $E$, but
that the $(s-1)$-ple $(c_1-c_2,\dots ,c_{s-1}-c_s)$ depends only from $E$. Indeed, a rank $r$ weakly uniform vector bundle $E$ of splitting type $(a_{h,i})$, $1 \le h \le r$,
$1 \le i \le s$, is uniform if and only if there are $s-1$ integers $d_j$, $2 \le j \le s$, such that $a_{h,i} = a_{h,1}+d_i$ for all $i\in \{2,\dots ,s\}$. If $E$ is uniform, then
the $r$-ples  $(a_{1,1}+y,\dots ,a_{r,1}+y)$, $y\in \mathbb {Z}$, are exactly the splitting types of $E$. If $E$ is uniform it is usually better to consider $E(0,a_{1,2}-a_{1,1},\dots ,a_{1,s}-a_{1,s})$ instead
of $E$, because all the splitting types of $E(0,a_{1,2}-a_{1,1},\dots ,a_{1,s}-a_{1,s})$ as a weakly uniform vector bundle are the same.

In this paper we prove the following result:

\begin{theorem}\label{i1}
Let $E$ be a rank $2$ vector bundle on $X$. $E$ is weakly uniform if and only if there are $L\in \mbox{Pic}(X)$, indices $1 \le i < j \le s$ and a rank $2$ weakly uniform vector bundle $G$ on $\mathbb {P}^{n_i}\times \mathbb {P}^{n_j}$
such that $E \otimes L \cong u_{ij}^\ast (G)$. $E$ splits if either $n_i\ge 3$ or $n_j \ge 3$. If $1\le n_1 \le 2$, $1 \le n_2 \le 2$ and $(n_1,n_2) \ne (1,1)$, then $E$ splits unless there is $h\in
\{1,2\}$ such that $n_h=2$ and
$E\otimes L \cong u_h^{\ast }(T\mathbb {P}^2)$ for some $L\in \mbox{Pic}(X)$.
\end{theorem}

Moreover we discuss the case of higher rank. We show that every rank $r>2$ weakly uniform vector bundle with splitting type $a_{1,1}=\dots =a_{r,s}=0$ is trivial and every rank $r>2$ uniform vector bundle with splitting type
 $a_1>\dots >a_r$, splits. Our methods did not allowed us to attack other splitting types.\\

\section{Weakly uniform rank two vector bundles}

In order to prove Theorem \ref{i1} we need a few lemmas.\\
We first consider the case $s=2$.

\begin{lemma}\label{a1}
Assume $s=2$, $n_1 =1$ and $n_2 = 2$. Let $E$ be a rank $2$ vector bundle on $\mathbb {P}^1 \times \mathbb {P}^2$. $E$ is
weakly uniform if and only if either $E$ splits as the direct sum of $2$ line bundles or there is a line bundle $L$ on
$\mathbb {P}^1
\times \mathbb {P}^2$ such that $E \cong L\otimes \pi _2^\ast (T\mathbb {P}^2)$.
\end{lemma}

\begin{proof} Since the ``~if~'' part is obvious, it is sufficient to prove the ``~only if~'' part. Let $(a_{h,i})$, $1\le h
\le 2$, $1 \le i \le s$, be the splitting type of $E$. Up to a twist by a line bundle we may assume $a_{1,1} = a_{1,2} = 0$.
By rigidity or
looking
at the Chern classes $c_i(E\vert \{Q\}\times \mathbb {P}^2)$, $i=1,2$, it is easy to see that if one of these two cases occurs for some $Q$, then
it occurs for all $Q$. First assume
$a_{2,2}=0$. Since the trivial line bundle on $\mathbb {P}^1$ is spanned, the theorem of changing basis implies that $F:= \pi
_{2\ast }(E)$ is a rank $2$ vector bundle on $\mathbb {P}^2$ and that the natural map
$\pi _2^\ast (F) \to E$ is an isomorphism (\cite{oss}, p. 11). Since $E$ is weakly uniform, $F$ is uniform. The classification of all rank $2$ uniform vector bundles on $\mathbb {P}^2$ shows that
either $F$ splits or it is isomorphic to a twist of $T\mathbb {P}^2$ (see \cite{ehs}), concluding the proof in the case $a_{2,2}=0$.
Similarly, if $a_{2,1}=0$, there is a rank $2$ vector bundle $G$ on $\mathbb {P}^1$ such that $\pi _1^\ast (G)\cong E$. Since every vector bundle on $\mathbb {P}^1$ splits, we have that also $E$ splits. Now we may assume $a_{2,2}<0$ and $a_{2,1} < 0$. Since
$a_{2,2}<0$, the base-change theorem gives that $\pi _{2 \ast} (E)$ is a line bundle, say of degree $b_2$, and that the natural map $\pi _2^\ast \pi _{2\ast }(E) \to E$ has locally free cokernel (\cite{oss}, p. 11). Thus in this case $E$ fits in an exact sequence
\begin{equation}\label{eqa2}
0 \to \mathcal {O}(0,b_2) \to E \to \mathcal {O}(a_{2,1},-b_2-a_{2,2}) \to 0
\end{equation}
The term $a_{2,1}$ in the last line bundle of (\ref{eqa2}) comes from $c_1(E)$. If (\ref{eqa2}) splits, then we are done.
Since $a_{2,1}\le 1$, K\"{u}nneth's formula gives
$H^1(\mathbb {P}^1\times \mathbb {P}^2,\mathcal {O}(-a_{2,1},2b_2+a_{2,2})) =0$. Hence
(\ref{eqa2}) splits.
 \end{proof}

  \begin{lemma}\label{a2}
Assume $s=2$, $n_1=1$ and $n_2\ge 3$. Then every rank two weakly uniform vector bundle on $X$ is the direct sum of two line bundles.
 \end{lemma}

 \begin{proof}
 We copy the proof of Lemma \ref{a1}. Every rank $2$ uniform vector bundle on $\mathbb {P}^m$, $m \ge 3$, splits. Hence $E$ splits even in the case $a_{2,2}=0$.
 \end{proof}

 \begin{lemma}\label{a3}
 Assume $s=2$ and $n_1=n_2=2$. Let $E$ be a rank $2$ indecomposable weakly uniform vector bundle on $X$. Then either $E \cong u_1^{\ast} (T\mathbb {P}^2)(u,v)$
 or $E \cong u_2^{\ast }(T\mathbb {P}^2)(u,v)$.
 \end{lemma}

 \begin{proof}
 Let $(a_{h,i})$ be the splitting type of $E$. Up to a twist by a line bundle we may assume $a_{1,1}=a_{1,2}=0$. As in the proof of Lemma \ref{a1} the theorem of changing basis
 gives that either $E \cong u_1^{\ast} (T\mathbb {P}^2(-2))$ or $E$ splits if $a_{2,1}=0$ and that $E \cong u_2^{\ast} (T\mathbb {P}^2(-2))$ or
$E$
splits if
$a_{2,2}=0$. If $a_{2,1}<0$ and $a_{2,2} <0$, then we apply
$\pi _{2\ast }$
and get an exact sequence (\ref{eqa2}). Here K\"{u}nneth's formula gives that (\ref{eqa2}) splits, without using any
information on the integer $a_{2,2}$.
 \end{proof}

 \begin{lemma}\label{a4}
 Assume $s=2$, $n_1 \ge 3$ and $n_2=2$. Let $E$ be a rank $2$ weakly uniform vector bundle on $X$. Then either $E$ splits or $E \cong u_2^{\ast }(T\mathbb {P}^2)(u,v)$
 for some integers $u, v$.
 \end{lemma}

 \begin{proof}
Let $(a_{hi})$ be the splitting type of $E$. Up to a twist by a line bundle we may assume $a_{1,1} = a_{1,2}=0$.  As in the proof of Lemma \ref{a1} the theorem of changing basis
 gives that $E \cong u_1^{\ast} (T\mathbb {P}^2(-2))$ or $E$ splits if $a_{2,1}=0$ and that $E$ splits in the case $a_{1,2}
<0$, because (\ref{eqa2}) splits by K\"{u}nneth's formula.
 \end{proof}

 \begin{lemma}\label{a5}
 Assume $s=2$, $n_1 \ge 3$ and $n_2\ge 3$. Let $E$ be a rank $2$ weakly uniform vector bundle on $X$. Then $E$ splits.
 \end{lemma}

 \begin{proof}
 Let $(a_{hi})$ be the splitting type of $E$. Up to a twist by a line bundle we may assume $a_{1,1} = a_{1,2}=0$. If $a_{2,2} =0$, then base change gives
 $E \cong u_2^\ast (F)$ for some uniform vector bundle on $\mathbb {P}^2$. Thus we may assume $a_{2,2}<0$. We have again the extension
 (\ref{eqa2}). Here again (\ref{eqa2}) splits by K\"{u}nneth's formula.
 \end{proof}
Now are ready to prove the main theorem:

 \vspace{0.3cm}

 \qquad {\emph {Proof of Theorem \ref{i1}.}} First assume
 $s=2$. Theorem \ref{i1} says nothing in the case $n_1=n_2=1$ for which a full classification is not known (\cite{bn} shows that moduli arises). Lemmas \ref{a1}, \ref{a2},
\ref{a3}, \ref{a4} and \ref{a5} cover all cases with $s=2$.
 Hence we may assume $s\ge 3$ and use induction on $s$. If $n_i=1$ for all $i$, then we may apply \cite{bn}, Theorem 4. For arbitrary $n_i$ the proof of \cite{bn}, Theorem 4, works verbatim, but for reader's
 sake we repeat that proof.
Let $(a_{hi})$ be the splitting type of $E$. Up to a twist by a line bundle we may assume $a_{1i}=0$ for all $i$. If $a_{2i}=0$ for some $i$, then the base-change theorem gives $E \cong \pi _i^\ast (F)$ for
some weakly uniform vector bundle $F$ on $X_i$. If $s=3$, then we are done. In the general case we reduce to the case $s' := s-1$. Thus to complete the proof it is sufficient either
to obtain a contradiction or to get that $E$ splits under the additional condition that $a_{2i} < 0$ for all $i$ and $s\ge 3$. Applying the base-change theorem to $\pi
_{1\ast}$ we get that $E$ fits in the following extension
\begin{equation}\label{eqa3}
0 \to \mathcal {O}(0,c_2,\dots ,c_s) \to E \to \mathcal {O}(a_{1,2},d_2,\dots ,d_s) \to 0
\end{equation}
Since $-a_{1,2} \ge 0$, K\"{u}nneth's formula shows that (\ref{eqa3}) splits unless $n_i=1$ for all $i\ge 2$. Using $\pi
_{2\ast }$ instead of $\pi _{1\ast }$ we get that $E$ splits, unless $n_1=1$.\qed

\section{Higher rank weakly uniform  vector bundles}
Now we consider higher rank weakly uniform  vector bundles.
                                                                                                        
\begin{proposition}\label{b1}
Let $E$ be a rank $r$ weakly uniform  vector bundle on $X$ with splitting type $(0,\dots ,0)$. Then $E$ is trivial.
\end{proposition}

\begin{proof}
The case $s=1$ is true by \cite{oss}, Theorem 3.2.1. Hence we may assume $s\ge 2$ and use induction on $s$. By the inductive assumption $E\vert \pi _1^{-1}(P)$ is trivial
for each $P\in \mathbb {P}^{n_1}$. By the base change theorem $F:= \pi_{1\ast }(E)$ is a rank $r$ vector bundle on $X_1$ and the natural map $\pi _1^\ast (F) \to E$ is an isomorphism.
This isomorphism implies that $F$ is uniform of splitting type $(0,\dots ,0)$. Hence the inductive assumption gives that $F$ is trivial. Thus $E$ is trivial.
\end{proof}
In order to study uniform vector bundles with $a_1>\dots >a_r$ we need the following lemmas:

\begin{lemma}\label{b2}
Fix an integer $r\ge 2$ and a rank $r$ vector bundle on $X$. Assume
the existence of an integer $i\in \{1,\dots ,s\}$ such that  $E\vert \pi _i^{-1}(P)$ is the direct sum of line
bundles for all $P\in X_i$. If $n_i=1$ assume that the splitting type of $E\vert \pi _i^{-1}(P)$
is the same for all $P\in X_i$. Let $(a_1,\dots ,a_r) = (b_1^{m_1},\dots ,b_k^{m_k})$, $b_1 > \cdots >b_k$, $m_1+\cdots +m_k=r$, be the splitting type
of $E\vert \pi ^{-1}(P)$ for any $P\in X_i$. Then there are $k$ vector bundles $F_1,\dots ,F_k$ on $X_i$ and $k$
vector bundles $E_1, \dots ,E_k$ on $X$ such that
$\mbox{rank}(F_i)=m_i$, $E_k=E$, $E_{i-1}$ is a subbundle of $E_i$ and $E_i/E_{i-1} \cong \pi _i^{\ast }(F_i)(-b_i)$ (with the convention $E_0=0$).
\end{lemma}

\begin{proof}
Notice that even in the case $n_i\ge 2$ the splitting type of  $E\vert \pi ^{-1}(P)$ does not depend from the choice of $P\in X_i$ (e.g. use Chern classes or local rigidity of direct sums of line bundles).
Thus $E\vert \pi _i^{-1}(P) \cong \oplus _{j=1}^{k} \mathcal {O}_{\pi _i^{-1}(P)}(b_j)^{\oplus m_j}$ for all $P\in X_i$. Set $F_1:= \pi _{j\ast }(E(0,\cdots ,-b_1,\cdot, 0)$. By the base-change theorem
$F_1$ is a rank $m_1$ vector bundle on $X_i$ and the natural map $\rho: \pi _i^\ast (F_1)(0,\cdots ,b_1,\dots )  \to E$ is a vector bundle embedding, i.e. either
$\rho$ is an isomorphism (case $r=m_1$) or $\mbox{Coker}(\rho)$ is a rank $r-m_1$ vector bundle on $X$. If $m_1=r$, then $k=1$ and we are won. Now assume
$k\ge 2$, i.e. $m_1<r$. Fix any $P\in X_i$. By definition $\mbox{Coker}(\rho )$ fits in an exact sequence of vector bundles on $X$:
\begin{equation}\label{eqa1}
0 \to \pi _i^\ast (F_1)(0,\dots ,b_1,\dots 0) \to E \to \mbox{Coker}(\rho )\to 0
\end{equation}
and the restriction to $\pi _i^{-1}(P)$ of the injective map of (\ref{eqa1}) induces an embedding of vector bundles $j_P: \mathcal {O}_{\pi _i^{-1}(P)}(b_1)^{\oplus m_1} \to  \oplus _{j=1}^{k} \mathcal {O}_{\pi _i^{-1}(P)}(b_j)^{\oplus m_j}$.
Since $b_1>b_j$ for all $j>1$, we get $\mbox{Coker}(j_P) \cong  \oplus _{j=2}^{k} \mathcal {O}_{\pi _i^{-1}(P)}(b_j)^{\oplus m_j}$. We apply to $\mbox{Coker}(\rho )$ the inductive assumption on $k$.
\end{proof}

\begin{lemma}\label{o1}
Assume $s=2$ and $n_1 \ge 2$, $n_2\ge 3$. Fix an integer $r$ such that $3 \le r\le n_2$ and a rank $r$ uniform vector bundle $E$ with splitting type $a_1>\cdots > a_r$.
Then $E$ is isomorphic to a direct sum of $r$ line bundles.
\end{lemma}

\begin{proof}
Since $r \ge 3$, we have $a_r \le a_1-2$. Thus the classification of uniform vector bundles on $\mathbb {P}^{n_2}$ with rank $r\le n_2$, gives
$E\vert \pi _1^{-1}(P) \cong \oplus _{i=1}^{r} \mathcal {O}_{\pi _1^{-1}(P)}(a_i)$ for all $P\in \mathbb {P}^{n_1}$. Apply Lemma \ref{b2}
with respect to the integers $i=1$ and $k=r$ and let $F_i, E_i$, $1 \le i \le r$, be the vector bundles given by the lemma.
Since $E_r = E$, it is sufficient to prove that each $E_i$ is a direct sum of $i$ line bundles. Since $\mbox{rank}(E_i) = i$, the
latter assertion is obvious if $i=1$. Fix an integer $i$ such that $1\le i < r$ and assume that $E_i$ is isomorphic to a direct sum of $i$ line bundles.
Lemma \ref{b2} gives an extension
$$ 0 \to E_i\to E_{i+1} \to L \to 0$$
with $L$ a line bundle on $\mathbb {P}^{n_1}\times \mathbb {P}^{n_2}$. Since $n_1\ge 2$ and $n_2\ge 2$, K\"unneth's formula gives that any extension of two line bundles on $\mathbb {P}^{n_1}\times \mathbb {P}^{n_2}$
splits. Thus $E_{i+1}$ is a direct sum of $i+1$ line bundles.
\end{proof}

\begin{proposition}\label{o2}
Fix an integer $r \ge 3$ and a rank $r$ uniform vector bundle on $X$ with splitting type $a_1>\cdots >a_r$. Assume $s \ge 2$, $n_2\ge r$ and $n_i\ge 2$ for all $i\ne 2$.
Then $E$ is isomorphic to a direct sum of $r$ line bundles.
\end{proposition}

\begin{proof}
The case $s=2$ is Lemma \ref{o1}. Thus we may assume $s\ge 3$ and that the proposition is true for $\mathbb {P}^{n_1}\times \cdots \times \mathbb {P}^{n_{s-1}}$.
By the inductive assumption $E\vert u_s^{-1}(P) \cong \oplus _{i=1}^{r} \mathcal {O}_{u_s^{-1}(P)}(a_i,\dots ,a_i)$ for all $P\in \mathbb {P}^{n_s}$.
As in the proof of Lemma \ref{b2} taking instead of $\pi _i$ the projection
$u_i: X \to \mathbb {P}^{n_i}$ we get line bundles $L_i$, $1 \le i \le r$ of $\mathbb {P}^{n_s}$,
(i.e. line bundles $u_i^\ast (L) \cong \mathcal {O}(0,\dots ,0,c_i,0,\cdots ,0)$ on $X)$ and subbundles $E_1\subset E_2\subset \cdots E_r=E$ such that
$E_i/E_{i-1} \cong \mathcal {O}_X(a_{i-1},\dots ,a_{i-1},c_i)$ (with the convention $E_0=0$). It is sufficient
to prove that each $E_i$ is isomorphic to a direct sum of $i$ line bundles. Since this is obvious for $i=1$, we may
use
induction on $i$. Fix an integer $i\in \{2,\dots ,r\}$. Our assumption on $X$ implies that the extension of any two line bundles splits.
Hence $E_i\cong E_{i-1} \oplus \mathcal {O}_X(a_{i-1},\dots ,a_{i-1},c_i)$.
\end{proof}

\providecommand{\bysame}{\leavevmode\hbox to3em{\hrulefill}\thinspace}

\end{document}